\documentclass[12pt]{article}
\usepackage[misc]{ifsym}
\usepackage[T1]{fontenc}
\usepackage[utf8]{inputenc}
\usepackage{authblk}
\usepackage{indentfirst}
\usepackage[top=0.8in, bottom=0.8in, left=0.9in, right=0.9in]{geometry}

\usepackage{amsmath}
\usepackage{enumerate}
\usepackage{amsthm}
\usepackage{abstract}

\usepackage{bm}
\usepackage{graphics}
\usepackage{graphicx}
\usepackage{subfigure}
\usepackage{caption}
\usepackage{amsthm,amsmath,amssymb}
\usepackage{float}
\usepackage[section]{placeins}
\usepackage{multirow}
\usepackage{amssymb}
\usepackage{IEEEtrantools}
\usepackage{cite}
\usepackage{pifont}
\pdfoutput=1

\title{Some results on the saturation number for unions of cliques}
\date{}
\author{Fan Chen$^{1,\, 2}$, Xiying Yuan$^{1,\, 2}$\thanks{Corresponding author. Email address: xiyingyuan@shu.edu.cn (Xiying Yuan)\\ \indent chenfan@shu.edu.cn (Fan Chen)\\ \indent This work is supported by the National Nature Science Foundation of China (Nos.11871040,
12271337, 12371347)}}
\affil{\small{\emph{1. Department of Mathematics, Shanghai University, Shanghai 200444, P.R. China}}\\
{\footnotesize\emph{2. Newtouch Center for Mathematics of Shanghai University, Shanghai 200444, P.R. China}}}

\begin{document}
\newtheorem{theorem}{Theorem}[section]
\newtheorem{assumption}[theorem]{Assumptio}
\newtheorem{corollary}[theorem]{Corollary}
\newtheorem{proposition}[theorem]{Proposition}
\newtheorem{lemma}[theorem]{Lemma}
\newtheorem{definition}[theorem]{Definition}
\newtheorem{remark}[theorem]{Remark}
\newtheorem{problem}[theorem]{Problem}
\newtheorem{fact}[theorem]{Fact}
\newtheorem{claim}{Claim}
\newtheorem{conjecture}[theorem]{Conjecture}
\maketitle
\noindent\rule[0pt]{17cm}{0.09em}

\noindent{\bf Abstract}

\noindent Graph $G$ is $H$-saturated if $H$ is not a subgraph of $G$ and $H$ is a subgraph of $G+e$ for any edge $e$ not in $G$. The saturation number for a graph $H$ is the minimal number of edges in any $H$-saturated graph of order $n$. In this paper, the saturation number for $K_p\cup (t-1)K_q$ ($t\geqslant 3$ and $2\leqslant p<q$) is determined, and the extremal graph for $K_p\cup 2K_q$ is determined. Moreover, the saturation number and the extremal graph for $K_p\cup K_q\cup K_r$ ($ r\geqslant p+q$) are completely determined.


\noindent\rule[0pt]{17cm}{0.05em}

\section{Introduction}
The vertex set and the edge set of a graph $G$ are denoted by $V(G)$ and $E(G)$. For any two graphs $G$ and $H$, $G\cup H$ is the union of the graph $G$ and $H$ with $V(G\cup H)=V(G)\cup V(H)$ and $E(G\cup H)=E(G)\cup E(H)$. The join of the graph $G$ and $H$, denoted by $G\vee H$, is the graph obtained from $G\cup H$ by adding edges between $V(G)$ and $V(H)$.

A graph $G$ is $H$-saturated if $H$ is not a subgraph of $G$ and $H$ is a subgraph of $G+e$ for any edge $e$ not in $G$. The saturation number for a graph $H$, denoted by $sat(n,H)$, is the minimal number of edges in any $H$-saturated graph of order $n$. In this paper, we call $G$ an extremal graph for $H$, if $G$ is an $H$-saturated graph of order $n$ with $e(G)=sat(n,H)$. Saturation numbers were first studied by P. Erd\H{o}s, A. Hajnal and J. W. Moon in \cite{PAJ}. In this paper, we mainly focus on the saturation numbers for unions of cliques. Denote a complete graph and an independent set of order $n$ by $K_n$ and $I_n$. In \cite{PAJ}, it was proved that $sat(n,K_r)=(r-2)(n-r+2)+\binom{r-2}{2}$ and the unique extremal graph for $K_r$ is $K_{r-2}\vee I_{n-r+2}$. L. K\'{a}szonyi and Z. Tuza in \cite{LZ} determined $sat(n,tK_2)$ and the extremal graph for $tK_2$. We refer the reader to \cite{C, CFFGJM, CLLYZ}, \cite{JRJ}, \cite{ HL}, and \cite{LGS, LSWZ, O, PS} for more results about the saturation number, and \cite{JRJ} is an excellent dynamic survey.

J. Faudree, R. Faudree and J. Schmitt suggested to investigate $sat(n,K_p\cup K_q\cup K_r)$ (see Problem 3 in \cite{JRJ}). The saturation number for $tK_p$ and the extremal graph for $3K_p$ have been determined in \cite{RMRM}. In this paper, we prove that
$$sat(n,K_p\cup (t-1)K_q)=(p-2)(n-p+2)+\binom{p-2}{2}+(t-1)\binom{q+1}{2},$$
where $2\leqslant p<q$, $n>q(q+1)(t-1)+3(p-2)$, and $t\geqslant 3$ (see Theorem \ref{ptq}).

Suppose that $2\leqslant p_1\leqslant \cdots\leqslant p_t$, and write
$$H(n;p_1,p_2,\cdots,p_t)\cong K_{p_1-2}\vee (K_{p_2+1}\cup \cdots\cup K_{p_t+1}\cup I_{n-t+3-\sum_{i=1}^tp_i}).$$
When $t=2$, it was proved that $H(n;p,q)$ is the extremal graph for $K_p\cup K_q$ (see Theorem 2.3 in \cite{RMRM}).
When $t=3$,
we will prove that $H(n;2,q,q)$ is the unique extremal graph for $K_p\cup 2K_q$ (see Theorem \ref{p2q}).
Consering the extremal graph for general $K_p\cup (t-1)K_q$, we propose the following conjecture.

\begin{conjecture}
Suppose $2\leqslant p\leqslant q$, $n>(t-1)(q+1)^2+3(p-2)$ and $t\geqslant 4$.
$H(n;p,q,\cdots,q)$ is the unique extremal graph for $K_p\cup (t-1)K_q$.
\end{conjecture}

When $2\leqslant p\leqslant q$ and $r\geqslant p+q$, we prove that
$$sat(n,K_p\cup K_q\cup K_{r})=(p-2)(n-p+2)+\binom{p-2}{2}+\binom{q+1}{2}+\binom{r+1}{2},$$
and $H(n;p,q,r)$ is the unique extremal graph for $K_p\cup K_q\cup K_{r}$ where $n>3(p-2)+q(q+1)+r(r+1)$ (see Theorem \ref{pqr}).

\section{Preliminaries}
For any vertex $v$ in $V(G)$, $N_G(v)$ is the set of the neighbors of $v$ and $N_G[v]=N_G(v)\cup\{v\}$. The degree of a vertex $v$ is $d_G(v)=|N_G(v)|$. If there is no ambiguity, then we write $N_G(v)=N(v)$ and $d_G(v)=d(v)$. Let the minimum degree of $G$ be $\delta(G)$. For $S\subseteq V(G)$, we write $\overline{S}=V(G)\setminus S$ and $G[S]$ is the subgraph of $G$ induced by the vertices in $S$. Let $G-S$ be the subgraph induced by $G[V(G)\setminus S]$. The order and size of $G$ are denoted by $v(G)$ and $e(G)$.

Next, we may prove several elementary properties common to all $K_{p_1}\cup \cdots\cup K_{p_t}$-saturated graph.

\begin{lemma}\label{detla}
Suppose $t\geqslant 2$, $n>3(p_1-2)+\sum_{i=2}^tp_i(p_i+1)$. Let $G$ be a $K_{p_1}\cup \cdots\cup K_{p_t}$-saturated graph and
$e(G)\leqslant e(H(n;p_1,\cdots,p_t)).$ Let $v$ be a vertex of minimum degree in $V(G)$. Then, we have

\noindent (i) $d(v)=p_1-2$.

\noindent (ii) $N(v)\subseteq N(w)$ for any $w\in \overline{N(v)}$.

\noindent (iii) $e(G[\overline{N(v)}])\leqslant\sum_{i=2}^t\binom{p_i+1}{2}.$

\end{lemma}
\begin{proof}[\rm{\textbf{Proof}.}]
We know that
\begin{equation}\label{a}
e(G)\leqslant e(H(n;p_1,\cdots,p_t))=(p_1-2)(n-p_1+2)+\binom{p_1-2}{2}+\sum_{i=2}^t\binom{p_i+1}{2}.
\end{equation}
Write
$a=(p_1-2)(n-p_1+2)+\binom{p_1-2}{2}+\sum_{i=2}^t\binom{p_i+1}{2}$.

Since $G$ is $K_{p_1}\cup\cdots\cup K_{p_t}$-free, $G\neq K_n$, and then $d_G(v)\leqslant n-2$. Since $G$ is $K_{p_1}\cup\cdots\cup K_{p_t}$-saturated, for any vertex $w$ in $\overline{N[v]}$, $G+vw$ contains a subgraph $K_{p_1}\cup \cdots\cup K_{p_t}$. Then the new edge $vw$ lies in a clique of order at least $p_1$ in $G+vw$. So, $d(v)=\delta(G)\geqslant p_1-2$. Hence, $w$ has at least $p_1-2$ neighbors in $N(v)$, which implies
$e(G[\overline{N(v)},N(v)])\geqslant (n-\delta-1)(p_1-2)+\delta.$
And there is a copy of $K_{p_1-2}$ in $G[N(v)]$, which implies
$e(G[N(v)])\geqslant \binom{p_1-2}{2}.$

To complete the proof of (i). Suppose to the contrary that $d(v)\geqslant p_1-1$. Write $\delta(G)=\delta$ for brevity. We have
\begin{align*}
\sum_{u\in N(v)}d_G(u)&=2e(G[N(v)])+e(G[\overline{N(v)},N(v)])\\
&\geqslant2\binom{p_1-2}{2}+(n-\delta-1)(p_1-2)+\delta.
\end{align*}
This yields that
\begin{align}\label{1}
2e(G)&=\sum_{u\in N(v)}d_G(u)+\sum_{u\in \overline{N(v)}}d_G(u)\notag\\
&\geqslant2\binom{p_1-2}{2}+(n-\delta-1)(p_1-2)+\delta+\delta(n-\delta).
\end{align}

Define the function
$$f(x)=2\binom{p_1-2}{2}+(n-x-1)(p_1-2)+x+x(n-x),$$
and then $f(\delta)\leqslant 2e(G)$ from (\ref{1}). By the assumption $n>3(p_1-2)+\sum_{i=2}^tp_i(p_i+1)$, we have
\begin{align*}
\frac{2a}{n}&=\frac{2\left[(p_1-2)(n-p_1+2)+\binom{p_1-2}{2}+\sum_{i=2}^t\binom{p_i+1}{2}\right]}{n}\\
&=2(p_1-2)+\frac{-p_1^2+4+3(p_1-2)+\sum_{i=2}^tp_i(p_i+1)}{n}\\
&<2(p_1-2)+\frac{n-p_1^2+4}{n}\\
&\leqslant2p_1-3\\
&\leqslant\frac{p_1(p_1+1)+2p_1-3}{2}\\
&\leqslant\frac{-p_1+3+3(p_1-2)+\sum_{i=2}^tp_i(p_i+1)}{2}\\
&<\frac{n-p_1+3}{2}.
\end{align*}
Observe that the axis of symmetry of $f(x)$ is $x=\frac{n-p_1+3}{2}$. Since
$$p_1-1\leqslant\delta \leqslant \frac{2e(G)}{n}\leqslant\frac{2a}{n}<\frac{n-p_1+3}{2},$$
we have $f(p_1-1)\leqslant f(\delta)$. Using equations (\ref{a}) and (\ref{1}), we get
$f(p_1-1) \leqslant f(\delta)\leqslant 2e(G)\leqslant 2a$, and then $f(p_1-1)\leqslant 2a$, i.e,
\begin{align*}
&2\binom{p_1-2}{2}+(n-p_1)(p_1-2)+(p_1-1)(n-p_1+1)\leqslant\\
&2(p_1-2)(n-p_1+2)+2\binom{p_1-2}{2}+\sum_{i=2}^tp_i(p_i+1).
\end{align*}
Hence, we have
$n\leqslant 3(p_1-2)+\sum_{i=2}^tp_i(p_i+1).$
This is a contradiction to the assumption $n>3(p_1-2)+\sum_{i=2}^tp_i(p_i+1)$. So, $\delta(G)=p_1-2$.

Now we prove part (ii). For any vertex $w$ in $\overline{N[v]}$, $w$ has at least $p_1-2$ neighbors in $N(v)$. Since $|N(v)|=p_1-2$, we have $N(v)\subseteq N(w)$.

Now we prove part (iii). Since $|N(v)|=p_1-2$, $vw$ lies in a clique of order $p_1$ in $G+vw$ and $G[N(v)]\cong K_{p_1-2}$. The fact $N(v)\subseteq N(w)$ for any $w\in \overline{N(v)}$ implies
$e(G[N(v),\overline{N(v)}])=(p_1-2)(n-p_1+2).$
By (\ref{a}), we get
\begin{align*}
e(G) &=e(G[N(v)])+e(G[N(v),\overline{N(v)}])+e(G[\overline{N(v)}])\\
&=\binom{p_1-2}{2}+(p_1-2)(n-p_1+2)+e(G[\overline{N(v)}])\\
&\leqslant \binom{p_1-2}{2}+(p_1-2)(n-p_1+2)+\sum_{i=2}^t\binom{p_i+1}{2}.
\end{align*}
Hence,
$e(G[\overline{N(v)}])\leqslant\sum_{i=2}^t\binom{p_i+1}{2}.$
\end{proof}
Suppose $t\geqslant 2$, $n>3(p_1-2)+\sum_{i=2}^tp_i(p_i+1)$. Let $G$ be a $K_{p_1}\cup \cdots\cup K_{p_t}$-saturated graph of order $n$ and $e(G)\leqslant e(H(n;p_1,\cdots,p_t))$. Let $v$ be a vertex of minimum degree in $G$. Write $S=N_G(v)$. For a vertex $w\in \overline{S}\setminus\{v\}$, $G+vw$ contains a subgraph $K_{p_1}\cup\cdots\cup K_{p_t}$. By Lemma \ref{detla} (i), $d_G(v)=p_1-2$ and then the new edge $vw$ lies in the copy of $K_{p_1}$ induced by $S\cup\{v,w\}$ in $G+vw$. Furthermore, there is a subgraph $K_{p_2}\cup \cdots \cup K_{p_t}$ in $G$. In this paper, we always use $H_{vw}$ to represent this subgraph $K_{p_2}\cup\cdots\cup K_{p_t}$. If $p_1=2$, then $\delta(G)=0$, and set $S=\emptyset$ and $\overline{S}=V(G)$.

For any vertex $u\in\overline{N[v]}$, if $u\sim w$, then $u\in V(H_{vw})$. By Lemma \ref{detla} (ii), we have $S\subseteq N_G(u)\cap N_G(w)$, and then $G[S\cup\{u,w\}]\cong K_{p_1}$. If $u\notin V(H_{vw})$, then $G[S\cup\{u,w\}]\cup H_{vw}$ is a subgraph $K_{p_1}\cup\cdots\cup K_{p_t}$ in $G$, which is a contradiction. The above facts can be organized as follows.

\begin{lemma}\label{re}
Suppose $t\geqslant 2$, $n>3(p_1-2)+\sum_{i=2}^tp_i(p_i+1)$. Let $G$ be a $K_{p_1}\cup \cdots\cup K_{p_t}$-saturated graph of order $n$ and $e(G)\leqslant e(H(n;p_1,\cdots,p_t))$. Let $v$ be a vertex of minimum degree in $G$ and $w\in \overline{S}\setminus\{v\}$.

\noindent (i) There is a subgraph $H_{vw}$ in $G$.

\noindent (ii) For any vertex $u\in\overline{S}\setminus\{v\}$, if $u$ is adjacent to $w$, then $u\in V(H_{vw})$.
\end{lemma}

\section{The saturation number for $K_p\cup (t-1)K_q$ ($2\leqslant p< q$)}

L. K\'asonyi and Z. Tuza determined $sat(n,tK_2)=3t-3$, and the extremal graph for $tK_2$ is $(t-1)K_3\cup (n-3t+3)K_1$ when $n\geqslant 3t-3$ (see Corollary 5 in \cite{LZ}). For general $tK_p$, R. Faudree, M. Ferrara, R. Gould, and M. Jacobson determined that
$$sat(n,tK_p)=(p-2)(n-p+2)+\binom{p-2}{2}+(t-1)\binom{p+1}{2},$$
and the extremal graph for $3K_p$ is $H(n;p,p,p)$ where $n\geqslant tp(p+1)-p^2+2p-6$ and $p\geqslant 3$ (see Theorem 2.1 and Theorem 2.3 in \cite{RMRM}). We will further consider the saturation number for $K_p\cup (t-1)K_q$ when $2\leqslant p<q$ and $t\geqslant 3$ in this section.
\begin{lemma}\label{saturated} $H(n;p,q,\cdots,q)$ is a $K_p\cup(t-1) K_q$-saturated graph.
\end{lemma}
\begin{proof}[\rm{\textbf{Proof}.}]
For brevity, write $G= H(n;p,q,\cdots,q)=G_1\vee(G_2\cup\cdots \cup G_t\cup I_{n-p-(t-1)(q+1)+2})$, where $G_1\cong K_{p-2}$ and $G_i\cong K_{q+1}$ for $2\leqslant i\leqslant t$.

There is not a subgraph $K_{p}$ in $G_1$, since $v(G_1)=p-2$. We can not find both $K_{p}$ and $K_{q}$ in any $G_1\vee G_i$, since $v(G_1\vee G_i)=p+q-1$. Hence, there is no subgraph $K_{p}\cup (t-1)K_{q}$ in $G_1\vee(G_2\cup\cdots \cup G_t)$. Then $G$ is $K_{p}\cup (t-1)K_{q}$-free.

For any non-edge $uv$ in $G$, then $u$ and $v$ lie in distinct $G_i$ or in $G_i$ and the independent set $I$ or in the independent set $I$. Observe that $G^*[V(G_1)\cup\{u,v\}]\cong K_{p}$ in the graph $G^*=G+uv$ and there is a subgraph isomorphic to $(t-1)K_{q}$ in $G-V(G_1)\cup\{u,v\}$. Hence, $G$ is $K_{p}\cup(t-1) K_{q}$-saturated.
\end{proof}

Lemma \ref{saturated} implies that
$sat(n,K_p\cup (t-1)K_q)\leqslant e(H(n;p,q,\cdots,q)).$
We may prove $sat(n,K_p\cup (t-1)K_q)=e(H(n;p,q,\cdots,q)).$
\begin{theorem}\label{ptq}
Suppose $2\leqslant p< q$, $n> q(q+1)(t-1)+3(p-2)$, and $t\geqslant 3$. Then
$$sat(n,K_p\cup (t-1)K_q)=(p-2)(n-p+2)+\binom{p-2}{2}+(t-1)\binom{q+1}{2}.$$
\end{theorem}
\begin{proof}[\rm{\textbf{Proof}.}]
By Lemma \ref{saturated}, we know $e(G)\leqslant H(n;p,q,\cdots,q)$. Let $G$ be an extremal graph for $K_p\cup (t-1)K_q$ of order $n>q(q+1)(t-1)+3(p-2)$. By Lemma \ref{re} (i), there is a subgraph $H_{vw}= H_1\cup \cdots \cup H_{t-1}$ with $H_i\cong K_q$ ($i=1,\cdots,t-1$) in $G$.

For any vertex $x\in V(H_i)$, we claim that $d_G(x)\geqslant p+q-2$. By Lemma \ref{detla} (ii), we have $S\subseteq N_G(x)$. For some vertex $x'\in (V(H_i)\setminus \{x\})$, by Lemma \ref{re} (i), we have $H_{vx'}\cong (t-1)K_q$ in $G$. By Lemma \ref{re} (ii), we have $x \in V(H_{vx'})$, and then $x$ is a clique of order $q$ in $V(H_{vx'})$. Combining the fact $S\cup \{x'\}\subseteq N_G(x)$. Hence, $d_G(x)\geqslant p-2+1+q-1=p+q-2$ for any vertex $x\in V(H_{vw})$.

Let $H':=G[V(H_{vw})]\setminus E(H_{vw})$ and
$R:=\big\{u\, |\, u\in \overline{S}\setminus V(H_{vw})\ \text{and}\ u\ \text{has at least one neighbor}$ $\text{in $V(H_{vw})$}\big\}.$
Next, we will show that $e(H')+e(G[V(H'),R])\geqslant v(H')$. Let $Q$ be a component of $H'$.

Suppose $Q$ contains a cycle. Then $e(Q)\geqslant v(Q)$. Hence,
\begin{equation}\label{e2}
e(Q)+e(G[V(Q),R])\geqslant v(Q).
\end{equation}

Suppose $Q$ is an isolated vertex, say $y$, and $y\in V(H_i)$. The facts $d_G(y)\geqslant p+q-2$ and $|S\cup V(H_i)\setminus \{y\}|=p+q-3$ ensure that $e(G[\{y\},R])\geqslant 1=v(Q)$. Hence, in this case, we have
\begin{equation}\label{e1}
e(Q)+e(G[V(Q),R])\geqslant v(Q).
\end{equation}

Now suppose $Q$ is a tree with at least two vertices. We may prove that some vertex of $Q$ has a neighbor in $R$.

Suppose to the contrary that $e(G[V(Q),R])=0$. Let $u$ be an end vertex of a longest path of $Q$. Without loss of generality, assume that $u\in V(H_1)$. Suppose that $z$ is the neighbor of $u$ in $Q$. Then $z\in V(H_i)$ and $i\geqslant 2$. Without loss of generality, assume that $z\in V(H_2)$. Write $V(H_1)=\{u,u_2,\cdots,u_q\}$. Since $u$ is a pendant vertex of $Q$ and $e(G[V(Q),R])=0$, we have $N_G[u]=S\cup V(H_1)\cup \{z\}$.

We claim that the vertex $z$ is adjacent to all vertices of $V(H_1)$. Without loss of generality, suppose to the contrary that $z\nsim u_2$. Considering the graph $G+vu_i$ for some $i\neq 2$, there is a subgraph $H_{vu_i}$ isomorphic to $(t-1)K_q$ in $G$ by Lemma \ref{re} (i).

 By Lemma \ref{re} (ii), we have $u\in V(H_{vu_i})$, namely $u$ is in a clique of $H_{vu_i}$. Noting that  $N_G[u]\setminus(S\cup \{u_i\})=(V(H_1)\setminus \{u_i\})\cup\{z\}$, then the vertex $u$ is in a clique of $H_{vu_i}$ induced by $(V(H_1)\setminus \{u_i\})\cup\{z\}$. But these vertices cannot be a clique since $z\nsim u_2$

So $V(H_1)\subseteq N_G(z)$. Since the pendant edge $uz$ lies on a longest path of $Q$, all but one vertex in $N_G(v)$ are the pendant vertices of $Q$. Without loss of generality, assume that $\{u,u_2,\cdots,u_{q-1}\}$ are pendant vertices of $Q$. Using the assumption $e(G[V(Q),R])=0$, we have $N_G[u_i]=S\cup V(H_1)\cup\{z\}$ for $2\leqslant i\leqslant q-1$.

Note that $N_G[u]=S\cup V(H_1)\cup\{z\}$. For a vertex $z'\in V(H_2)\setminus\{z\}$, we have $u\nsim z'$, and $|N_G(u)\cap N_G(z')|\leqslant p$. If $|N_G(u)\cap N_G(z')|= p$, then $N_G(u)\cap N_G(z')= S\cup \{u_q,z\}$. The graph $G_3=G+uz'$ contains a subgraph $H_{uz'}^*$ isomorphic to $K_p\cup (t-1)K_q$. Obviously, the new edge $uz'$ lies in $H_{uz'}^*$. Let $W$ be the maximal clique containing the edge $uz'$ in $H_{uz'}^*$. Then $v(W)=p$ or $q$.

The fact $|N_G(u)\cap N_G(z')|\leqslant p$ implies that $v(W)\leqslant p+2$. Now we distinguish three cases to obtain contradictions by finding a subgraph $K_p\cup (t-1)K_q$ in $G$. Let $W'=H_{uz'}^*- V(W)$.

\textbf{Case 1.} $v(W)=p+2$.

In this case, $W=G_3[S\cup\{u,z',z,u_q\}]$, and then $v(W)=p+2=q$.

If $u_2\notin V(W')$, then $G[S\cup\{u,u_2,z,u_q\}]\cup W'$ is a subgraph $K_p\cup (t-1)K_q$ in $G$. So $u_2$ is in $W'$. Note that $N_G[u_2]\setminus (S\cup\{u,z,u_q\})=\{u_2,\cdots,u_{q-1}\}$. Hence, $u_2$ is in a clique $K_p$ of $W'$, and this clique is $G[\{u_2,\cdots,u_{q-1}\}]$. Therefore, $G[S\cup\{z,z'\}]\cup G[\{u,u_2,\cdots,u_q\}]\cup (W'-\{u_2,\cdots,u_{q-1}\})$ is a subgraph $K_p\cup (t-1)K_q$ in $G$.

\textbf{Case 2.} $v(W)=p+1$.

Noting that $V(W)\setminus \{u,z'\}\subseteq S\cup \{z,u_q\}$, we distinguish two subcases according to the vertices in $V(W)\setminus\{u,z'\}$.

\textbf{Subcase 2.1.} $V(W)\setminus\{u,z'\}=S\cup\{z\}$ or $V(W)\setminus\{u,z'\}=S\cup\{u_q\}.$

Without loss of generality, assume that $V(W)\setminus\{u,z'\}=S\cup\{z\}$. Then $W=G_3[S\cup \{u,z',z\}]$, and then $v(W)=p+1=q$.

If $u_2\notin V(W')$, then $G[S\cup\{u,u_2,z\}]\cup W'$ is a subgraph $K_p\cup (t-1)K_q$ in $G$. So $u_2$ is in $W'$.
Note that $N_G[u_2]\setminus (S\cup\{u,z\})=\{u_2,\cdots,u_q\}$. Hence, $u_2$ is in a clique $K_p$ of $W'$, and this clique is $G[\{u_2,\cdots,u_q\}]$. Therefore, $G[S\cup\{z,z'\}]\cup G[\{u,u_2,\cdots,u_q\}]\cup (W'-\{u_2,\cdots,u_q\})$ is a subgraph $K_p\cup (t-1)K_q$ in $G$.

\textbf{Subcase 2.2.} $V(W)\setminus\{u,z'\}=S'\cup\{z,u_q\}$ with $S'\subseteq S$ and $|S'|=p-3$ when $p\geqslant 3$.

Then $W=G_3[S'\cup\{u,z',z,u_q\}]$, and then $v(W)=p+1=q$.

If $u_2\notin V(W')$, then $G[S'\cup\{u,u_2,z,u_q\}]\cup W'$ is a subgraph $K_p\cup (t-1)K_q$ in $G$. So $u_2$ is in $W'$. Note that $N_G[u_2]\setminus (S\cup\{u,z,u_q\})=\{u_2,\cdots,u_{q-1}\}$. Write $\{s\}=S\setminus S'$. Hence, $u_2$ is in a clique $K_p$ of $W'$, and this clique is $G[\{u_2\cdots,u_{q-1},s\}]$. Therefore, $G[S\cup\{z,z'\}]\cup G[\{u,u_2,\cdots,u_q\}]\cup (W'-\{u_2,\cdots,u_{q-1},s\})$ is a subgraph $K_p\cup (t-1)K_q$ in $G$.

\textbf{Case 3.} $v(W)=p$.

Noting that $V(W)\setminus\{u,z'\} \subseteq S\cup \{z,u_q\}$, we distinguish three subcases according to the vertices in $V(W)\setminus\{u,z'\}$.

\textbf{Subcase 3.1.} $V(W)\setminus\{u,z'\}=S$.

In this subcase, $W=G_3[S\cup \{u,z'\}]\cong K_p$.

If $u_2\notin V(W')$, then $G[S\cup\{u,u_2\}]\cup W'$ is a subgraph $K_p\cup (t-1)K_q$ in $G$. So $u_2$ is in $W'$. Note that $N_G[u_2]\setminus (S\cup \{u\})=\{u_2,\cdots,u_q,z\}$. Hence, $u_2$ is in a clique $K_q$ of $W'$, and this clique is $G[\{u_2,\cdots,u_q,z\}]$. Thus $G[S\cup\{z,z'\}]\cup G[\{u,u_2,\cdots,u_q\}]\cup (W'-\{u_2,\cdots,u_q,z\})$ is a subgraph $K_p\cup (t-1)K_q$ in $G$.

\textbf{Subcase 3.2.} $V(W)\setminus\{u,z'\}=S'\cup\{z\}$ or  $V(W)\setminus\{u,z'\}=S'\cup\{u_q\}$ with $S'\subseteq S$ and $|S'|=p-3$ when $p\geqslant 3$.

Without loss of generality, assume that $V(W)\setminus\{u,z'\}=S'\cup\{z\}$. Then $W=G[S'\cup \{u,z',z\}]\cong K_p$.

If $u_2\notin V(W')$, then $G[S'\cup\{u,u_2,z\}]\cup W'$ is a subgraph $K_p\cup (t-1)K_q$ in $G$. So $u_2$ is in $W'$.
Write $\{s\}=S\setminus S'$. Note that $N_G[u_2]\setminus (S'\cup \{u,z\})=\{u_2,\cdots,u_q,s\}$. Hence, $u_2$ is in a clique $K_q$ of $W'$, and this clique is $G[\{u_2,\cdots,u_q,s\}]$. Thus $G[S\cup\{z,z'\}]\cup G[\{u,u_2,\cdots,u_q\}]\cup (W'-\{u_2,\cdots,u_q,s\})$ is a subgraph $K_p\cup (t-1)K_q$ in $G$.

\textbf{Subcase 3.3.} $V(W)\setminus\{u,z'\}=S''\cup\{z,u_q\}$ with $S''\subseteq S$ and $|S''|=p-4$ when $p\geqslant 4$.

We have $W=G[S''\cup \{u,z',z,u_q\}]\cong K_p$.

If $u_2\notin V(W')$, then $G[S''\cup \{u,u_2,z,u_q\}]\cup W'$ is a subgraph $K_p\cup (t-1)K_q$ in $G$. So $u_2$ is in $W'$. Write $\{s_1,s_2\}=S\setminus S''$. Note that $N_G[u_2]\setminus (S''\cup \{u,z,u_q\})=\{u_2,\cdots,$ $u_{q-1},s_1,s_2\}$. Hence, $u_2$ is in a clique $K_q$ of $W'$, and this clique is $G[\{u_2,\cdots,$ $u_{q-1},s_1,s_2\}]$. Thus $G[S\cup\{z,z'\}]\cup G[\{u,u_2,\cdots,u_q\}]\cup (W'- \{u_2,\cdots,u_{q-1},s_1,s_2\})$ is a subgraph $K_p\cup (t-1)K_q$ in $G$.

The above cases show that $e(G[V(Q),R])\geqslant 1$ for a tree $Q$, and then
\begin{equation}\label{e3}
e(Q)+e(G[V(Q),R])\geqslant v(Q).
\end{equation}
Using equations (\ref{e2}), (\ref{e1}), and (\ref{e3}), we have
$e(H')+e(G[V(H'),R])\geqslant v(H')=v(H)$. Furthermore,
\begin{align*}
e(G[\overline{S}])&\geqslant e(H_{vw})+e(H')+e(G[V(H),R])\\
&\geqslant e(H_{vw})+v(H')\\
&= (t-1)\binom{q}{2}+(t-1)q\\
&=(t-1)\binom{q+1}{2}.
\end{align*}
By Lemma \ref{detla} (iii), we have
$e(G[\overline{S}])\leqslant(t-1)\binom{q+1}{2}$.
Hence,
$e(G[\overline{S}])=(t-1)\binom{q+1}{2},$
and then
\begin{align*}
e(G)=e(G[S])+e(G[S,\overline{S}])+e(G[\overline{S}])=\binom{p-2}{2}+(p-2)(n-p+2)+(t-1)\binom{q+1}{2}.
\end{align*}
This completes the proof.
\end{proof}

\begin{theorem}\label{p2q}
Suppose $2\leqslant p< q$ and $n> 2q(q+1)+3(p-2)$. Then $H(n;p,q,q)$ is the unique extremal graph for $K_p\cup 2K_q$.
\end{theorem}
\begin{proof}[\rm{\textbf{Proof}.}]
Let $G$ be an extremal graph for $K_p\cup 2K_q$ of order $n> 2q(q+1)+3(p-2)$. By Lemma \ref{re} (i), there is a subgraph $H_{vw}=H_{vw,1}\cup H_{vw,2}$ with $H_{vw,1}\cong H_{vw,2}\cong K_q$ in $G$. Write $V(H_{vw,1})=\{u_1,\cdots,u_q\}$ and $V(H_{vw,2})=\{v_1,\cdots,v_q\}$. By Theorem \ref{ptq}, we have
\begin{equation}\label{2q}
e(G[\overline{S}])-e(H_{vw})= 2\binom{q+1}{2}-2\binom{q}{2}=2q.
\end{equation}

We claim that $e(G[\overline{S}\setminus V(H_{vw})])=0$. Otherwise, there is an edge, say $e=xy$, in $E(G[\overline{S}\setminus V(H_{vw})])$. Then $G[S\cup\{x,y\}]\cup H_{vw}\cong K_p\cup 2K_q$, which is a contradiction. Hence, for any $x\in V(G)$, $N_G(x)\subseteq S\cup V(H_{vw})$. Let
$$R:=\big\{u\, |\, u\in \overline{S}\setminus V(H_{vw})\ \text{and}\ u \ \text{has at least one neighbor \text{in} $V(H_{vw})$}\big\}.$$
Recall that $S\subseteq N_G(x)$ for any vertex $x$ of $G$. Then, in $G$, the vertices of degree at least $p-1$ are in $S\cup V(H_{vw})\cup R$.

Now we will show that $|R|\geqslant2$.
There is some vertex $u_i$ is not adjacent to $v_j$ for $1\leqslant i,j\leqslant q$. Otherwise, we have $e(G[\overline{S}])-e(H_{vw})\geqslant q^2>2q$,
which is a contradiction to (\ref{2q}). Without loss of generality, assume that $u_1\nsim v_1$. Then the graph $G+u_1v_1$ contains $K_p\cup 2K_q$ as a subgraph. If $|R|\leqslant 1$, then
$|S\cup V(H_{vw})\cup R|\leqslant p-2+2q+1<p+2q.$
There is not a subgraph $K_p\cup 2K_q$ in $G+u_1v_1$. So $|R|\geqslant 2$ holds.

We claim that $d_{H_{vw}}(x)=q$ for any vertex $x\in R$. Let $y$ be a neighbor of $x$ in $V(H_{vw})$ for some vertex $x\in R$. By Lemma \ref{re} (i), there is a subgraph $H_{vy}$ isomorphic to $2K_q$ in $G$. By Lemma \ref{re} (ii), we have $x\in V(H_{vy})$, and then $x$ is a clique of order $q$ in $V(H_{vy})$. Recall that $N_G(x)\subseteq S\cup V(H_{vw})$. Noting that $y\in N_G(x)$ and $y\notin V(H_{vy})$, we have $d_{H_{vw}}(x)\geqslant q$ for any vertex $x\in R$.
Using equation (\ref{2q}), we have
\begin{equation}\label{3}
|R|q\leqslant\sum_{x\in R}d_{H_{vw}}(x)= e(G[R,V(H_{vw})])\leqslant e(G[\overline{S}])-e(H_{vw})= 2q.
\end{equation}
Then $|R|=2$. And (\ref{3}) become equalities. We have $d_{H_{vw}}(x)=q$, and then $d_G(x)=p+q-2$ for any vertex $x\in R$. Furthermore, we have $e(G[\overline{S}])- e(H_{vw})=e(G[R,V(H_{vw})])$. So $u_i\nsim v_j$ for any $1\leqslant i,j\leqslant q$.

Let $R=\{x_1,x_2\}$, and then $x_1\nsim x_2$ since $E(G[\overline{S}])= E(H_{vw})\cup E(G[R,V(H_{vw})])$. Without loss of generality, assume that $u_1$ is a neighbor of $x_1$ in $H_{vw}$. Considering the graph $G+vu_1$, there is a subgraph $H_{vu_1}$ isomorphic to $2K_q$ in $G$.
By Lemma \ref{re} (ii), we have $x_1\in V(H_{vu_1})$, namely $x_1$ is in a clique of order $q$ in $H_{vu_1}$ .

Note that $|N_G(x_1)\setminus (S\cup \{u_1\})|=q-1$, and then vertices in $N_G(x_1)\setminus (S\cup \{u_1\})$ induce a clique of order $q-1$. Since $N_G(x_1)\setminus(S\cup \{u_1\})\subseteq V(H_{vw})$, we have $N_G(x_1)\setminus(S\cup \{u_1\})\subseteq V(H_{vw,1})$ or $N_G(x_1)\setminus(S\cup \{u_1\})\subseteq V(H_{vw,2})$.

We will show that $N_G(x_1)\setminus(S\cup \{u_1\})\subseteq V(H_{vw,1})$. Suppose to the contrary that $N_G(x_1)\setminus(S\cup \{u_1\})\subseteq V(H_{vw,2})$. Then we may assume that $x_1\sim v_i$ ($i=1,\cdots,q-1$), and then $N_G(x_1)=S\cup\{u_1,v_1,\cdots,v_{q-1}\}$. By considering the auxiliary graph $G+vv_1$, the vertex $x_1$ is in a clique of order $q$ in $H_{vv_1}$. Furthermore, this clique is induced by $\{u_1,v_2,\cdots,v_{q-1},x_1\}$, which is a contradiction to the fact $u_1\nsim v_i$ for any $1\leqslant i\leqslant q$. Hence, $N_G(x_1)\setminus(S\cup \{u_1\})\subseteq V(H_{vw,1})$.

Since $d_G(x_1)=p+q-2$, we have $N_G(x_1)=S\cup V(H_{vw,1})$. Similarly, we may prove $N_G(x_2)=S\cup V(H_{vw,i})$ for some $1\leqslant i\leqslant2$.

We claim that $N_G(x_2)=S\cup V(H_{vw,2})$. Otherwise, suppose to the contrary that $N_G(x_2)=S\cup V(H_{vw,1})$. Then $v_2$ has no neighbor in $R$. So $N_G[v_2]=S\cup V(H_{vw,2})$, and then $|N_G[v_2]\setminus(S\cup\{v_1\})|=|V(H_{vw,2})\setminus\{v_1\}|= q-1$. By Lemma \ref{re} (ii), we have $v_2\in V(H_{vv_1})$. Then we have $|N_G[v_2]\setminus(S\cup\{v_1\})|\geqslant q$, which is a contradiction. Thus, $N_G(x_1)=S\cup V(H_{vw,1})$ and $N_G(x_2)=S\cup V(H_{vw,2})$. By Lemma \ref{detla} (ii),
we know $S\subseteq N_G(w)$ for any $w\in \overline{N(v)}$, and then
\begin{align*}
G&=G[S]\vee\left(G[V(H_{vw,1})\cup\{x_1\}]\cup G[V(H_{vw,2})\cup\{x_2\}]\cup I_{n-p-2q}\right)\\
&\cong K_{p-2}\vee(K_{q+1}\cup K_{q+1}\cup I_{n-p-2q})
\end{align*}
Hence, $G\cong H(n;p,q,q)$.
\end{proof}

\section{The saturation number for $K_p\cup K_q\cup K_r$ ($r\geqslant p+q$)}
The saturation number and the extremal graph for $K_p\cup 2K_q$ have been determined in Section 3, and then we always suppose that $2\leqslant p\leqslant q\leqslant r-1$ in this section. In Lemma \ref{7}, we will prove that $H(n;p,q,r)$ is $K_p\cup K_q\cup K_r$-saturated when $r\geqslant p+q$. We may point out that the condition $r\geqslant p+q$ is necessary in Lemma \ref{7}. In fact, when $2\leqslant p\leqslant q\leqslant r-1$ and $r\leqslant p+q-1$, $H(n;p,q,r)$ contains a subgraph $K_p\cup K_q\cup K_r$. When $p=2$, and then $r=q+1$, noting that $H(n;2,q,q+1)\cong K_{q+1}\cup K_{q+2}\cup I_{n-2q-3}$. We may find a $K_2\cup K_q$ in $K_{q+2}$. Hence, $K_{q+1}\cup K_{q+2}$ contains a subgraph $K_2\cup K_q\cup K_{q+1}$. When $p\geqslant 3$ and $r\leqslant p+q-1$, we may write $K_{p-2}$ as $K_{p+q-r-1}\vee K_{r-q-1}$. Then the graph $K_{p+q-r-1}\vee K_{r+1}$ contains a subgraph $K_p\cup K_q$, and $K_{r-q-1}\vee K_{q+1}=K_r$ holds. Hence, $K_{p-2}\vee (K_{q+1}\cup K_{r+1})$ contains a subgraph $K_p\cup K_q\cup K_r$.

\begin{lemma}\label{7}
$H(n;p,q,r)$ is $K_p\cup K_q\cup K_r$-saturated when $r\geqslant p+q$.
\end{lemma}
\begin{proof}[\rm{\textbf{Proof}.}]
We may prove that $H(n;p,q,r)$ is $K_p\cup K_q\cup K_r$-free. Suppose to the contrary that there is a subgraph $K_p\cup K_q\cup K_r$ in $H(n;p,q,r)$. Since $p+q-1<r$, the graph $K_{p-2}\vee K_{q+1}$ does not contain a subgraph $K_r$. Hence, $K_r$ is in the graph $K_{p-2}\vee K_{r+1}$. Let $R$ be the set of the vertices of $K_r$. Since there are $p-1$ vertices in $V(K_{p-2}\vee K_{r+1})\setminus R$, there are $p+q$ vertices in $V(K_{p-2}\vee(K_{q+1}\cup K_{r+1}))\setminus R$, and then these vertices induce a $K_p\cup K_q$. Note that the vertices of $K_{r+1}$ and the vertices of $K_{q+1}$ are non-adjacent. Hence, there is at least one vertex in $V(K_{p-2}\vee K_{r+1})\setminus R$ which is not adjacent to the vertices in $K_{q+1}$. This is a contradiction. Hence, $H(n;p,q,r)$ is $K_p\cup K_q\cup K_r$-free when $r\geqslant p+q$. After adding an edge in $H(n;p,q,r)$, there is a subgraph $K_p\cup K_q\cup K_r$. Hence, $H(n;p,q,r)$ is $K_p\cup K_q\cup K_r$-saturated when $r\geqslant p+q$.
\end{proof}

In this section, we will further prove that $H(n;p,q,r)$ is the extremal graph for $K_p\cup K_q\cup K_r$ when $r\geqslant p+q$.

\begin{theorem}\label{pqr}
Suppose $2\leqslant p\leqslant q$, $r\geqslant p+q$, and $n>3(p-2)+q(q+1)+r(r+1)$. Then
$$sat(n,K_p\cup K_q\cup K_r)=(p-2)(n-p+2)+\binom{p-2}{2}+\binom{q+1}{2}+\binom{r+1}{2}.$$
Furthermore, $H(n;p,q,r)$ is the unique extremal graph for $K_p\cup K_q\cup K_r$.
\end{theorem}

Let $G$ be an extremal graph for $K_p\cup K_q\cup K_r$ of order $n>3(p-2)+q(q+1)+r(r+1)$ when $r\geqslant p+q$. By Lemma \ref{7}, we have $e(G)\leqslant e(H(n;p,q,r))$. Let $v$ be a vertex of minimum degree in $G$. By Lemma \ref{detla} (i), we have $d(v)=p-2$. Write $S=N_G(v)$ and $\overline{S}=V(G)\setminus N_G(v)$. If $p=2$, then we set $S=\emptyset$ and $\overline{S}=V(G$). By Lemma \ref{detla} (iii), we have
\begin{equation}\label{qr}
e(G[\overline{S}])\leqslant\binom{q+1}{2}+\binom{r+1}{2}.
\end{equation}

By Lemma \ref{re} (i), we know $H_{vw}= H_{vw,1}\cup H_{vw,2}$ with $H_{vw,1}\cong K_q$ and $H_{vw,2}\cong K_r$ is a subgraph of $G$. Write $V(H_{vw,1})=\left\{u_1, \cdots, u_q\right\}$ and $V(H_{vw,2})=\left\{v_1, \cdots, v_r\right\}$.

Let $H$ be an auxiliary graph with $V(H)=V(H_{vw})\cup V(H_{vu_1})$ and $E(H)=E(H_{vw})\cup E(H_{vu_1})$ (see Figure \ref{H}). For $1\leqslant i,j\leqslant 2$, write $V(H_{vu_1,i})\cap V(H_{vw,j})=L_{ij}$ and $|L_{ij}|=\ell_{ij}$.
\begin{figure}[h]
  \centering
  \includegraphics[width=0.5\textwidth]{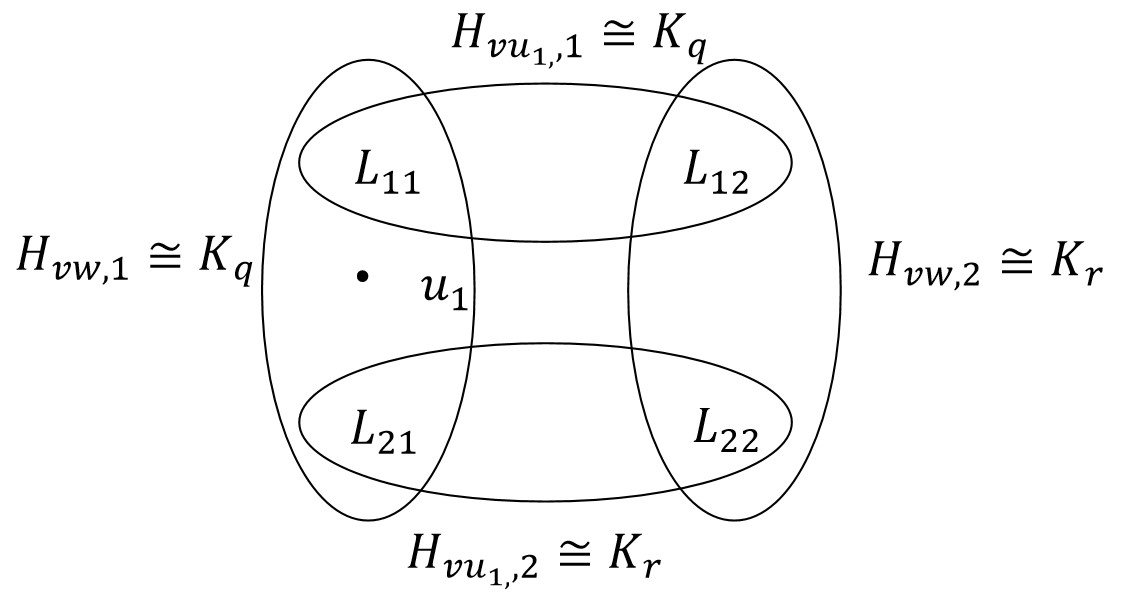}
  \caption{Auxiliary graph $H$.}\label{H}
\end{figure}
Then we have
\begin{align*}
e(H)=& e(H_{vw,1})+e(H_{vw,2})+e(H_{vu_1,1})+e(H_{vu_1,2})-\sum_{i,j=1}^2e(H[L_{ij}]) \notag\\
=&\binom{q}{2}+\binom{r}{2}+\binom{q}{2}+\binom{r}{2}-\sum_{i,j=1}^2\binom{\ell_{ij}}{2} \notag\\
=&\binom{q+1}{2}+\binom{r+1}{2}+f(\ell_{11},\ell_{12},\ell_{21},\ell_{22}),
\end{align*}
where
\begin{align*}
f(\ell_{11},\ell_{12},\ell_{21},\ell_{22})=&\frac{1}{2}\left[(q^2-3q)+(r^2-3r)-(\ell_{11}+\ell_{12})^2-(\ell_{21}+\ell_{22})^2\right.\\
 & \left.+2(\ell_{11}\ell_{12}+\ell_{21}\ell_{22})+(\ell_{11}+\ell_{12}+\ell_{21}+\ell_{22})\right].
\end{align*}
Note $e(G[\overline{S}])=e(H)+e(G[\overline{S}]\setminus E(H))$. Hence,
\begin{equation}\label{S-}
e(G[\overline{S}])=\binom{q+1}{2}+\binom{r+1}{2}+f(\ell_{11},\ell_{12},\ell_{21},\ell_{22})+e(G[\overline{S}]\setminus E(H)),
\end{equation}
and
\begin{equation}\label{6}
e(G[\overline{S}])\geqslant\binom{q+1}{2}+\binom{r+1}{2}+f(\ell_{11},\ell_{12},\ell_{21},\ell_{22}).
\end{equation}

We claim that $\ell_{11}+\ell_{12}\geqslant q-1$. If $\ell_{11}+\ell_{12}\leqslant q-2$, then there is an edge, say $xy$, in $G[V(H_{vu_1})\setminus V(H_{vw})]$. Hence, $G[S\cup\{x,y\}]\cup H_{vw}$ is a subgraph isomorphic to $K_p\cup K_q\cup K_r$ in $G$, which is a contradiction. Similarly, $\ell_{21}+\ell_{22}\geqslant r-1$ and $\ell_{12}+\ell_{22}\geqslant r-1$ hold.

Since $u_1\in V(H_{vw,1})$ and $u_1\notin V(H_{vu_1,1})$, $\ell_{11}+\ell_{21}\leqslant q-1$ holds. By Lemma \ref{re} (ii), we have $V(H_{vw,1})\setminus\{u_1\}\subseteq V(H_{vu_1})$. Furthermore by the definition of $L_{11}$ and $L_{21}$, we have $V(H_{vw,1})\setminus\{u_1\}\subseteq L_{11}\cup L_{21}$, so $\ell_{11}+\ell_{21}\geqslant q-1$. Then we have $\ell_{11}+\ell_{21}= q-1$.

So we have the following bounds for $\ell_{ij}$.
\begin{equation}\label{k1112}
\left\{
\begin{array}{ll}
      q-1\leqslant\ell_{11}+\ell_{12}\leqslant q; & \hbox{} \\
      r-1\leqslant\ell_{21}+\ell_{22}\leqslant r; & \hbox{} \\
      \ell_{11}+\ell_{21}= q-1; & \hbox{} \\
      r-1\leqslant\ell_{12}+\ell_{22}\leqslant r. & \hbox{}
    \end{array}
\right.
\end{equation}

Now we will make effort to prove $\ell_{12}=0$ (see Lemma \ref{nonthing}) and $\ell_{11}=q-1$ (see Lemma \ref{q-1}), which are key points to determine the extremal graph for $K_p\cup K_q\cup K_r$.
\begin{lemma}\label{5}
$|V(H_{vu_1,1})\cap V(H_{vw,2})|\leqslant 1$.
\end{lemma}
\begin{proof}[\rm{\textbf{Proof}.}]
Suppose to the contrary that $|V(H_{vu_1,1})\cap V(H_{vw,2})|=\ell_{12}\geqslant 2$. By (\ref{k1112}), $\ell_{11}\leqslant q-2$ and $\ell_{22}\leqslant r-2$. Hence, we have
$r-1\leqslant \ell_{21}+\ell_{22}\leqslant \ell_{21}+(r-2).$
Then $\ell_{21}\geqslant 1$ holds.  By $\ell_{21}\leqslant q-1$, we have
$r-1\leqslant \ell_{21}+\ell_{22}\leqslant (q-1)+\ell_{22}.$
Hence, $\ell_{22}\geqslant r-q\geqslant p\geqslant 2$ holds.

Since $\ell_{11}+\ell_{12}+\ell_{21}+\ell_{22}\leqslant q+r-1$, the case $\ell_{11}+\ell_{12}=q$ and $\ell_{21}+\ell_{22}=r$ does not appear. We distinguish three cases to obtain contradictions.

\noindent\textbf{Case 1.} $\ell_{11}+\ell_{12}= q-1$ and $\ell_{21}+\ell_{22}= r-1$.

In this case, $f(\ell_{11},\ell_{12},\ell_{21},\ell_{22})=\ell_{11}\ell_{12}+\ell_{21}\ell_{22}-2.$ Since $\ell_{21}\geqslant 1$ and $\ell_{22}\geqslant 2$, we have $\ell_{21}\ell_{22}\geqslant r-2$. Hence,
$$f(\ell_{11},\ell_{12},\ell_{21},\ell_{22})\geqslant 0+(r-2)-2=r-4.$$
The assumption $\ell_{12}\geqslant 2$ implies that $q=\ell_{11}+\ell_{12}+1\geqslant 3$, and then $r\geqslant p+q\geqslant 5$. From (\ref{6}), we know
$$e(G[\overline{S}])\geqslant\binom{q+1}{2}+\binom{r+1}{2}+r-4>\binom{q+1}{2}+\binom{r+1}{2},$$
which is a contradiction to (\ref{qr}).

\noindent\textbf{Case 2.} $\ell_{11}+\ell_{12}= q-1$ and $\ell_{21}+\ell_{22}=r$.

In this case, $f(\ell_{11},\ell_{12},\ell_{21},\ell_{22})=\ell_{11}\ell_{12}+\ell_{21}\ell_{22}-r-1.$
Recall that $\ell_{21}\leqslant q-1$ and $\ell_{22}\leqslant r-2$. Then we have
$r=\ell_{21}+\ell_{22}\leqslant \ell_{21}+(r-2),$
and
$r=\ell_{21}+\ell_{22}\leqslant (q-1)+\ell_{22}.$
Hence, $\ell_{21}\geqslant 2$ and $\ell_{22}\geqslant 3$. Then $\ell_{21}\ell_{22}\geqslant 2(r-2).$ Hence,
$$f(\ell_{11},\ell_{12},\ell_{21},\ell_{22})\geqslant 0+2(r-2)-r-1=r-5.$$
The assumption $\ell_{12}\geqslant 2$ implies that $q=\ell_{11}+\ell_{12}+1\geqslant 3$, and then $r\geqslant p+q\geqslant 5$. By (\ref{S-}), we have
$$e(G[\overline{S}])\geqslant\binom{q+1}{2}+\binom{r+1}{2}+r-5+e(G[\overline{S}]\setminus E(H))\geqslant\binom{q+1}{2}+\binom{r+1}{2}.$$
Then, by (\ref{qr}), we have
$e(G[\overline{S}])=\binom{q+1}{2}+\binom{r+1}{2}$, which implies $e(G[\overline{S}]\setminus E(H))=0$. Hence, $u_1\nsim v_1$ and the vertex of degree at least $p-1$ are in $S\cup V(H)$. Since $v(H)=q+r+1$, we have $|S\cup V(H)|=p+q+r-1$, which contradicts to the assumption $K_p\cup K_q\cup K_r\subseteq G+u_1v_1$.

\noindent\textbf{Case 3.} $\ell_{11}+\ell_{12}=q$ and $\ell_{21}+\ell_{22}= r-1$.

In this case, $f(\ell_{11},\ell_{12},\ell_{21},\ell_{22})=\ell_{11}\ell_{12}+\ell_{21}\ell_{22}-q-1.$ Since $\ell_{21}+\ell_{22}= r-1$, there is a vertex $y\in V(H)\setminus V(H_{vw})$ such that $G[L_{21}\cup L_{22}\cup\{y\}]\cong K_r$.

If $\ell_{11}\geqslant1$, then $q=\ell_{11}+\ell_{12}\geqslant 3$. Since $\ell_{12}\geqslant 2$, $\ell_{21}\geqslant 1$, and $\ell_{22}\geqslant 2$, we have $\ell_{11}\ell_{12}\geqslant q-1$ and $\ell_{21}\ell_{22}\geqslant r-2$. Then
$$f(\ell_{11},\ell_{12},\ell_{21},\ell_{22})\geqslant (q-1)+(r-2)-q-1=r-4.$$
By $q\geqslant 3$, $r\geqslant q+2\geqslant 5$ holds. From (\ref{6}), we have
$$e(G[\overline{S}])\geqslant\binom{q+1}{2}+\binom{r+1}{2}+r-4>\binom{q+1}{2}+\binom{r+1}{2},$$
which is a contradiction to (\ref{qr}).

So, we may suppose $\ell_{11}=0$. Then $\ell_{12}=q$. By (\ref{k1112}), we know that $\ell_{21}= q-1$, and then $\ell_{22}=r-q$. Then we have
$f(\ell_{11},\ell_{12},\ell_{21},\ell_{22})=f(0,\,q,\,q-1,\,r-q)=-q^2+(q-1)r-1.$

Without loss of generality, assume that $L_{12}=\{v_1,\cdots,v_q\}$. We claim that there exists some vertex, say $v_1$, in $\{v_1,\cdots,v_q\}$ such that, in $G$, $u_1\nsim v_1$ and $N_G(u_1)\cap N_G(v_1)=S$. Otherwise, we have $u_1\sim v_j$ or $ S\subset N_G(u_1)\cap N_G(v_j)$ for any $1\leqslant j \leqslant q$, which implies that $e(G[\overline{S}]\setminus E(H))\geqslant q$, and then by (\ref{S-})
\begin{equation*}
e(G[\overline{S}])\geqslant\binom{q+1}{2}+\binom{r+1}{2}-q^2+(q-1)r+q-1>\binom{q+1}{2}+\binom{r+1}{2}.
\end{equation*}

So, $u_1\nsim v_1$ and $N_G(u_1)\cap N_G(v_1)= S$ in $G$. Considering the graph $G_1=G+u_1v_1$, then there is a subgraph $K_p\cup K_q\cup K_r$ in $G_1$. Furthermore, $G_1[S\cup\{u_1,v_1\}]\cong K_p$, since $N_G(u_1)\cap N_G(v_1)=S$. Noting that $|S\cup V(H)|=p+q+r-1$, there is a vertex $x$ in $\overline{S}\setminus V(H)$ such that $x$ has at least $q-1$ neighbors in $G[\overline{S}]$, which implies that $e(G[\overline{S}]\setminus E(H))\geqslant q-1$. Then by (\ref{S-})
\begin{equation*}
e(G[\overline{S}])\geqslant\binom{q+1}{2}+\binom{r+1}{2}-q^2+(q-1)r+q-2\geqslant\binom{q+1}{2}+\binom{r+1}{2},
\end{equation*}
with equality holding if and only if $q=2$ and $r=4$. When $q=2$ and $r=4$, we have $p=2$, and then $S=\emptyset$ and $\overline{S}=V(G)$. Furthermore, $x$ has unique neighbor in $G$. Since there is a subgraph $K_2\cup K_2\cup K_4$ in $G_1=G+u_1v_1$, we have $x$ is adjacent to $v_2$. Then, $e(G)=13$ and $E(G)= E(H_{vw})\cup\{xv_2,u_2v_3,u_2v_4,yu_2,yv_3,yv_4\}$. Let $F_1$ be the graph shown in Figure 2. Then $G= F_1\cup I_{n-8}$, while there is not a subgraph $K_2\cup K_2\cup K_4$ in $G+xv_1$.
\end{proof}
\begin{figure}
  \centering
  \begin{minipage}{0.15\linewidth}
  \centering
  \includegraphics[width=\textwidth]{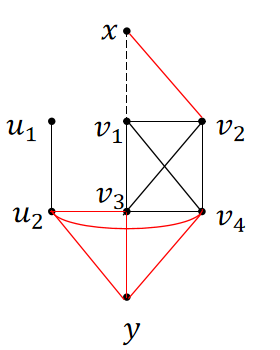}
  \caption*{\footnotesize{$F_1$}}
  \end{minipage}
  \begin{minipage}{0.15\linewidth}
  \centering
  \includegraphics[width=\textwidth]{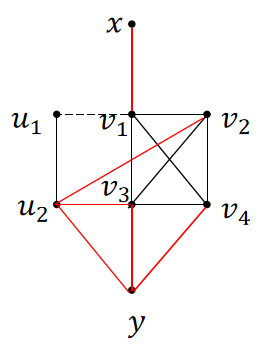}
  \caption*{\footnotesize{$F_2$}}
  \end{minipage}
    \begin{minipage}{0.15\linewidth}
  \centering
  \includegraphics[width=\textwidth]{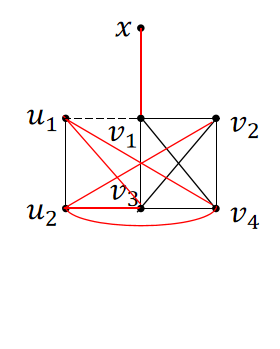}
  \caption*{\footnotesize{$F_3$}}
  \end{minipage}
    \begin{minipage}{0.15\linewidth}
  \centering
  \includegraphics[width=\textwidth]{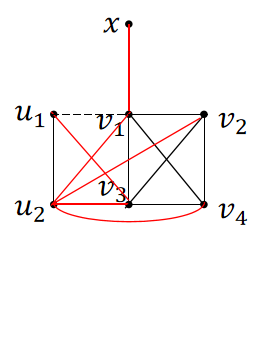}
  \caption*{\footnotesize{$F_4$}}
  \end{minipage}
    \begin{minipage}{0.15\linewidth}
  \centering
  \includegraphics[width=\textwidth]{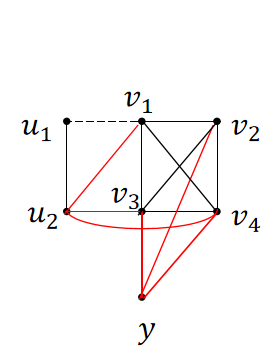}
  \caption*{\footnotesize{$F_5$}}
  \end{minipage}
  \caption{Graphs $F_1$, $F_2$, $F_3$, $F_4$, and $F_5$.}
\end{figure}

\begin{lemma}\label{nonthing1}
If $\ell_{12}=1$, then $\ell_{11}\neq q-2$.
\end{lemma}
\begin{proof}[\rm{\textbf{Proof}.}]

Without loss of generality, assume $L_{12}=\{v_1\}$. Suppose to the contrary that $\ell_{11}= q-2$. Set $L_{11}=\emptyset$ when $q=2$. Without loss of generality, assume $L_{11}=\{u_2,\cdots,u_{q-1}\}$ when $q\geqslant 3$. By $\ell_{11}+\ell_{21}=q-1$, we have $\ell_{21}=1$, and then $L_{21}=\{u_q\}$. Since $\ell_{11}+\ell_{12}=q-1$, there is a vertex $x\in V(H)\setminus V(H_{vw})$ such that $G[L_{11}\cup L_{12}\cup\{x\}]\cong K_q$. Furthermore, we have $r-2\leqslant\ell_{22}\leqslant r-1$ by (\ref{k1112}).

\noindent\textbf{Case 1.} $\ell_{22}=r-2$.

Without loss of generality, assume $L_{22}=\{v_2,\cdots,v_{r-1}\}$. Since $\ell_{21}+\ell_{22}=r-1$, there is a vertex $y\in V(H)\setminus V(H_{vw})$ such that $G[L_{21}\cup L_{22}\cup\{y\}]\cong K_r$. In this case, we have $f(\ell_{11},\ell_{12},\ell_{21},\ell_{22})=f(q-2,\,1,\,1,\,r-2)=q+r-6.$ By (\ref{S-}) and the facts $q\geqslant 2$ and $r\geqslant p+q\geqslant 4$, we have
$$e(G[\overline{S}])\geqslant\binom{q+1}{2}+\binom{r+1}{2}+q+r-6\geqslant\binom{q+1}{2}+\binom{r+1}{2},$$
with equality holding if and only if $q=2$ and $r=4$. When $q=2$ and $r=4$, we have $p=2$, and then $S=\emptyset$ and $\overline{S}=V(G)$. Furthermore, $e(G)=13$ and $E(G)=E(H_{vw})\cup \{u_2v_2,u_2v_3,xv_1,yu_2,yv_2,yv_3\}$. Let $F_2$ be the graph shown in Figure 2. Then $G=F_2\cup I_{n-8}$, while there is not a subgraph $K_2\cup K_2\cup K_4$ in $G+u_1v_1$.

\noindent\textbf{Case 2.} $\ell_{22}=r-1$.

The assumption $\ell_{22}=r-1$ implies that $L_{22}=\{v_2,\cdots,v_r\}$. In this case, we have $f(\ell_{11},\ell_{12},\ell_{21},\ell_{22})=f(q-2,\,1,\,1,\,r-1)=q-4.$ By (\ref{S-}), we have
\begin{equation}\label{8}
e(G[\overline{S}])=\binom{q+1}{2}+\binom{r+1}{2}+e(G[\overline{S}]\setminus E(H))+q-4.
\end{equation}

Considering the graph $G_2=G+vv_2$, $G_2[S\cup \{v,v_2\}]\cong K_p$ since $N_G(v)\cap N_G(v_2)=S$, and there is a subgraph $K_r$ in $G[\overline{S}]-\{v_2\}$. Let $U=\{u_q,v_1,v_3,\cdots,v_r\}$ and $U'=V(H)\setminus (U\cup \{v_2\})$, namely, $U'=\{u_1,\cdots,u_{q-1},x\}$.

\begin{claim}\label{cl1}
If there is a clique of order $r$ in $G[\overline{S}]-\{v_2\}$, then this clique is induced by $U$.
\end{claim}
\begin{proof}[\rm{\textbf{Proof of Claim \ref{cl1}}.}]
Suppose to the contrary that there is a subgraph $W\cong K_r$ such that $V(W)\neq U$ in $G[\overline{S}]-\{v_2\}$. Let $U_0=V(W)\cap U$ and $U_0'=V(W)\setminus U$. Write $|U_0'|=k$. Then $e(W[U_0,U_0'])=k(r-k)$.

Note that $U_0'=V(W)\setminus U\subseteq \overline{S}\setminus (U\cup \{v_2\})=U'\cup (\overline{S}\setminus V(H))$. We claim that $U_0'\subseteq U'$. Otherwise, there is a vertex $z\in \overline{S}\setminus V(H)$ such that $z\in U_0' \subseteq V(W)$. Hence, there are at least $r-1$ edges in $E(G[\overline{S}])\setminus E(H)$. By (\ref{8}), we have
$$e(G[\overline{S}])\geqslant \binom{q+1}{2}+\binom{r+1}{2}+q+r-5>\binom{q+1}{2}+\binom{r+1}{2}.$$

Then $U_0'\subseteq U'$. Hence, $1\leqslant k\leqslant q\leqslant r-p\leqslant r-2$.

When $k=r-2$, we have $p=2$, $q=r-2$, and $e(W[U_0,U_0'])=2(r-2)$. Furthermore, $S=\emptyset$, $\overline{S}=V(G)$, $|U_0|=2$, and $U_0'= \{u_1,\cdots,u_{q-1},x\}$. The fact $u_1$, $x\in U_0'\subseteq V(W)$ implies that $u_1x\in E(W)$. So $u_1x\in E(W)\setminus E(H)$. Note that $E(H[U_0,U_0'])\subseteq E(H[U,U_0'])=E(H[\{u_q,v_1\}, U_0'])$. If $u_q$ or $v_1\notin U_0$, without loss of generality, assume that $v_1\notin U_0$, then $e(H[U_0,U_0'])=e(H[\{u_q\},U_0'])\leqslant r-2$. By (\ref{8}), we have
$$e(G\setminus E(H))\geqslant e(W[U_0,U_0'])-e(H[U_0,U_0'])+1\geqslant 2(r-2)-(r-2)+1\geqslant r-1.$$
By (\ref{8}), we have
$$e(G)\geqslant \binom{q+1}{2}+\binom{r+1}{2}+q+r-5>\binom{q+1}{2}+\binom{r+1}{2}.$$

Hence, $U_0=\{u_q,v_1\}$, and then $u_q\sim v_1$. Note that $u_1\sim u_q$, $u_1\nsim v_1$, $x\nsim u_q$, and $x\sim v_1$ in $H$. Hence, $\{u_qv_1,u_1v_1,xu_1,xu_q\}\subseteq E(G)\setminus E(H)$. Then $e(G\setminus E(H))\geqslant 4$. By (\ref{8}), we have
$$e(G)\geqslant \binom{q+1}{2}+\binom{r+1}{2}+q>\binom{q+1}{2}+\binom{r+1}{2}.$$
Hence, $1\leqslant k\leqslant r-3$.

Note that $E(H[U_0,U_0'])\subseteq E(H[U,U_0'])=E(H[\{u_q,v_1\}, U_0'])$. If $|\{u_q,v_1\}\cap U_0|\leqslant 1$, then $e(H[U_0,U_0'])\leqslant k$. Hence,
$$e(G[\overline{S}]\setminus E(H))\geqslant e(W[U_0,U_0'])-e(H[U_0,U_0'])= k(r-k)-k\geqslant k(r-1-k)\geqslant r-2.$$
By (\ref{8}), we have
$$e(G[\overline{S}])\geqslant \binom{q+1}{2}+\binom{r+1}{2}+q+r-6\geqslant\binom{q+1}{2}+\binom{r+1}{2},$$
with the equality holding if and only if $k=1$, $q=2$, and $r=4$. Furthermore, $e(H[U_0,U_0'])=k$, and then there is a vertex of $\{u_q,v_1\}$ in $U_0$. Without loss of generality, assume $u_q\in U_0$ and $U_0'=\{u_1\}$.  When $q=2$ and $r=4$, we have $p=2$, and then $S=\emptyset$ and $\overline{S}=V(G)$. Furthermore, $E(G)=E(H_{vw})\cup\{u_1v_3,u_1v_4,u_2v_2,u_2v_3,u_2v_4,xv_1\}$. Let $F_3$ be the graph shown in Figure 2. Then $G=F_3\cup I_{n-7}$, while there is not a subgraph $K_2\cup K_2\cup K_4$ in $G+u_1v_1$.

If $\{u_q,v_1\}\subseteq U_0$, then $u_qv_1\in E(G[\overline{S}])\setminus E(H)$, and $e(H[U_0,U_0'])= e(H[\{u_q,v_1\},U_0'])\leqslant 2k$. Hence,
$$e(G[\overline{S}]\setminus E(H))\geqslant e(W[U_0,U_0'])-e(H[U_0,U_0'])+1\geqslant k(r-k)-2k+1=k(r-2-k)+1\geqslant r-2.$$
By (\ref{8}), we have
$$e(G[\overline{S}])\geqslant \binom{q+1}{2}+\binom{r+1}{2}+q+r-6\geqslant\binom{q+1}{2}+\binom{r+1}{2},$$
with equality holding if and only if $k=1$, $q=2$, and $r=4$. When $k=1$, without loss of generality, assume $U_0=\{u_1\}$. When $q=2$ and $r=4$, we have $p=2$, and then $S=\emptyset$ and $\overline{S}=V(G)$. Furthermore, $E(G)=E(H_{vw})\cup\{u_1v_3,u_2v_1,u_2v_2,u_2v_3,u_2v_4,xv_1\}$. Let $F_4$ be the graph shown in Figure 2. Then $G=F_4\cup I_{n-7}$, while there is not a subgraph $K_2\cup K_2\cup K_4$ in $G+u_1v_1$.
\end{proof}

Recall that there is a subgraph $K_p\cup K_q\cup K_r$ in $G_2=G+vv_2$ and $G_2[S\cup \{v,v_2\}]\cong K_p$. Hence, there is a clique of order $r$ in $G[\overline{S}]-\{v_2\}$. By Claim \ref{cl1}, $G[U]$ is the unique clique of order $r$ in the graph $G[\overline{S}]-\{v_2\}$, and then $u_q\sim v_1$. We may prove that there is not a subgraph $K_q$ in $G[\overline{S}]- (U\cup\{v_2\})$.

Suppose to the contrary that there is a subgraph $K_q$ in $G[\overline{S}]- (U\cup\{v_2\})$. We claim that the graph induced by $U'$ is not a $K_q$. Otherwise, $G[U']\cong K_q$, and then $xu_1\in E(G[\overline{S}])\setminus E(H)$. Recall that $u_qv_1\in E(G[\overline{S}])\setminus E(H)$. Then $e(G[\overline{S}]\setminus E(H))\geqslant 2$. By (\ref{8}), we have
$$e(G[\overline{S}])\geqslant \binom{q+1}{2}+\binom{r+1}{2}+q-2\geqslant\binom{q+1}{2}+\binom{r+1}{2},$$
with equality holding if and only if $q=2$. When $q=2$, we have $p=2$, and then $S=\emptyset$ and $\overline{S}=V(G)$. Furthermore, $E(G)=E(H_{vw})\cup \{xu_1,xv_1,u_2v_1,\cdots,u_2v_r\}$, and then the vertices of degree at least one are in $V(H)$. Note that $v(H)=r+3$ and there are two nonadjacent vertices $u_i$ and $v_j$ in $V(H)$, otherwise $G[V(H)]$ is a clique of order $r+3$, and then $e(G)>\binom{3}{2}+\binom{r+1}{2}$. While the fact $v(H)=r+3$ implies that there is not a subgraph $K_2\cup K_2\cup K_r$ in $G+u_iv_j$.

Since the graph induced by $U'$ is not the $K_q$, there is at least one vertex $z'\in \overline{S}\setminus V(H)$ in the clique of order $q$. Furthermore, there is only one vertex $z'\in \overline{S}\setminus V(H)$ in the clique of order $q$. Otherwise, there are two vertices $z'$ and $z''\in \overline{S}\setminus V(H)$ in the clique of order $q$. Then $G[S\cup\{z',z''\}]\cup H_{vw}$ is a subgraph $K_p\cup K_q\cup K_r$ in $G$. Hence, there are $q-1$ vertices of $\{u_1,\cdots,u_{q-1},x\}$ in the clique of order $q$. Denote this clique by $W'$. Since $u_1\nsim x$, we have $u_1\in V(W')$ or $x\in V(W')$.

If $u_1\in V(W')$, then $z'\sim u_1$. By Lemma \ref{re} (ii), we have $z'\in V(H_{vu_1})$, which is a contradiction to $z'\in \overline{S}\setminus V(H)$.

If $x\in V(W')$, then $z'\sim x$. Hence, $G[S\cup \{z',x\}]\cup H_{vw}$ is a subgraph $K_p\cup K_q\cup K_r$ in $G$.
\end{proof}

\begin{lemma}\label{nonthing2}
If $\ell_{12}=1$, then $\ell_{11}\neq q-1$.
\end{lemma}
\begin{proof}[\rm{\textbf{Proof}.}]
Suppose to the contrary that $\ell_{11}= q-1$, which implies that $L_{11}=\{u_2,\cdots,u_q\}$. Without loss of generality, assume $L_{12}=\{v_1\}$. By (\ref{k1112}), we have $\ell_{21}=0$ and $\ell_{22}=r-1$. Furthermore, we have
$f(\ell_{11},\ell_{12},\ell_{21},\ell_{22})=f(q-1,\,1,\,0,r-1)=-2.$
By (\ref{S-}), we have
\begin{equation}\label{9}
e(G[\overline{S}])=\binom{q+1}{2}+\binom{r+1}{2}+e(G[\overline{S}]\setminus E(H))-2.
\end{equation}

Since $\ell_{22}=r-1$, there is a vertex, say $y\in V(H)\setminus V(H_{vw})$ such that $G[L_{22}\cup \{y\}]\cong K_r$ in $G$. There is some vertex in $\{v_2,\cdots,v_r,y\}$, say $v_2$, such that $u_1\nsim v_2$ and $N_G(u_1)\cap N_G(v_2)=S$. Otherwise, we have $e(G[\overline{S}]\setminus E(H))\geqslant r$, and then by (\ref{9})
$$e(G[\overline{S}])\geqslant \binom{q+1}{2}+\binom{r+1}{2}+r-2>\binom{q+1}{2}+\binom{r+1}{2}.$$

Considering the graph $G_1=G+u_1v_2$,  $G_1[S\cup\{u_1,v_2\}]\cong K_p$ since $N_G(u_1)\cap N_G(v_2)=S$, and there is a subgraph $K_r$ in $G[\overline{S}]-\{u_1,v_2\}$. Let $U=\{v_1,v_3,\cdots,v_r,y\}$ and $U'=V(H)\setminus (U\cup\{u_1,v_2\})$, namely, $U'=\{u_2,\cdots,u_q\}$.

~\\

~\\

\begin{claim}\label{cl2}
If there is a clique of order $r$ in $G[\overline{S}]-\{u_1,v_2\}$, then this clique is induced by $U$.

\end{claim}
\begin{proof}[\rm{\textbf{Proof of Claim \ref{cl2}}.}]
Suppose to the contrary that there is a subgraph $W\cong K_r$ such that $V(W)\neq U$ in $G[\overline{S}]-\{u_1,v_2\}$. Let $U_0=V(W)\cap U$ and $U_0'=V(W)\setminus U$. Write $|U_0'|=k$. Then $e(W[U_0,U_0'])=k(r-k)$.

Note that $U_0'=V(W)\setminus U\subseteq \overline{S}\setminus (U\cup \{u_1,v_2\})=U'\cup (\overline{S}\setminus V(H))$. We claim that $U_0'\subseteq U'$. Otherwise, there is a vertex $z\in \overline{S}\setminus V(H)$ such that $z\in U_0' \subseteq V(W)$. Hence, there are at least $r-1$ edges in $E(G[\overline{S}])\setminus E(H)$. By (\ref{9}), we have
$$e(G[\overline{S}])\geqslant \binom{q+1}{2}+\binom{r+1}{2}+r-3>\binom{q+1}{2}+\binom{r+1}{2}.$$

Then $U_0'\subseteq U'$. Hence, $1\leqslant k\leqslant q-1\leqslant r-p-1\leqslant r-3$. Note that $E(H[U_0,U_0'])\subseteq E(H[U,U_0'])=E(H[\{v_1\}, U_0'])$, we have $e(H[U_0,U_0'])= e(H[\{v_1\},U_0'])\leqslant k$, and then
$$e(G[\overline{S}]\setminus E(H))\geqslant e(W[U_0,U_0'])-e(H[U_0,U_0'])\geqslant k(r-k)-k=k(r-1-k).$$
Note that $k(r-1-k)\geqslant r-2$. By (\ref{9}), we have
$$e(G[\overline{S}])\geqslant \binom{q+1}{2}+\binom{r+1}{2}+r-4\geqslant\binom{q+1}{2}+\binom{r+1}{2},$$
with equality holding if and only if $k=1$ and $r=4$. When $k=1$, $U_0'=\{u_2\}$. When $r=4$, we have $p=2$ and $q=2$, and then $S=\emptyset$ and $\overline{S}=V(G)$. Furthermore, $E(G)=E(H_{vw})\cup\{u_2v_1,u_2v_3,u_2v_4,yv_2,yv_3,yv_4\}$. Let $F_5$ be the graph shown in Figure 2. Then $G=F_5\cup I_{n-7}$, while there is not a subgraph $K_2\cup K_2\cup K_4$ in $G+u_1v_1$.
\end{proof}

Recall that there is a subgraph $K_p\cup K_q\cup K_r$ in $G_1=G+u_1v_2$ and $G_1[S\cup\{u_1,v_2\}]\cong K_p$. Hence, there is a clique of order $r$ in $G[\overline{S}]-\{u_1,v_2\}$. By Claim \ref{cl2}$, G[U]$ is the unique clique of order $r$ in $G[\overline{S}]-\{u_1,v_2\}$, and then $y\sim v_1$. We may prove that there is not a subgraph $K_q$ in $G[\overline{S}]- (U\cup \{u_1,v_2\})$.

Suppose to the contrary that there is a subgraph $K_q$ in $G[\overline{S}]- (U\cup \{u_1,v_2\})$. Note that $|V(H)\setminus\{u_1,v_1,\cdots,v_r,y\}|=|\{u_2,\cdots,u_q\}|=q-1$. Hence, there is at least one vertex $z'\in \overline{S}\setminus V(H)$ in a clique of order $q$. Furthermore, there is only one vertex $z'\in \overline{S}\setminus V(H)$ in a clique of order $q$. Otherwise, there are two vertices $z'$ and $z''\in \overline{S}\setminus V(H)$. Then $G[S\cup\{z',z''\}]\cup H_{vw}$ is a subgraph $K_p\cup K_q\cup K_r$ in $G$. Hence, $G[\{z',u_2,\cdots,u_q\}]\cong K_q$. Then there are $q-1$ edges in $E(G[\overline{S}])\setminus E(H)$. Recall that $yv_1\in E(G[\overline{S}])\setminus E(H)$. By (\ref{9}), we have
$$e(G[\overline{S}])\geqslant \binom{q+1}{2}+\binom{r+1}{2}+q-2\geqslant\binom{q+1}{2}+\binom{r+1}{2},$$
with equality holding if and only if $q=2$. When $q=2$, we have $p=2$, and then $S=\emptyset$ and $\overline{S}=V(G)$. Furthermore, $N_G(z')=\{u_2\}$ and $E(G)=E(H_{vw})\cup \{u_2v_1,yv_1,\cdots,yv_r,z'u_2\}$. Hence, the vertices of degree at least one are in $V(H)\cup \{z'\}$.

Note that $N_G(u_2)\cap N_G(v_2)=\{v_1\}$ and $G[\{v_1,v_3,\cdots,v_r,y\}]$ is the unique clique of order $r$ in $G-\{u_1,v_2\}$. Since $u_1\nsim v_j$ for $1\leqslant j\leqslant r$ and $u_1\nsim y$, $G[\{v_1,v_3,\cdots,v_r,y\}]$ is the unique clique of order $r$ in $G-\{v_2\}$. Furthermore, in $G_2=G+u_2v_2$, we have $G_2[\{u_2,v_2\}]\cong K_2$. Note that $N_G(z')\setminus \{u_2\}=\emptyset$. Hence, $z'$ is not in the subgraph $K_2\cup K_2\cup K_r$ in $G_2$. Since $v(H)=r+3$, there is not a subgraph $K_2\cup K_2\cup K_r$ in $G_2$.
\end{proof}

\begin{lemma}\label{nonthing}
$|V(H_{vu_1,1})\cap V(H_{vw,2})|=0$.
\end{lemma}
\begin{proof}[\rm{\textbf{Proof}.}]
By Lemma \ref{5}, suppose to the contrary that $\ell_{12}= 1$. By (\ref{k1112}), we have $q-2 \leqslant\ell_{11}\leqslant q-1$. By Lemmas \ref{nonthing1} and \ref{nonthing2}, $\ell_{11}\neq q-2$ and $\ell_{11}\neq q-1$. Hence, $\ell_{12}= 0$.
\end{proof}
By Lemma \ref{re} (i), $G$ has a subgraph $H_{vu_i}= H_{vu_i,1}\cup H_{vu_i,2}$ with $H_{vu_i,1}\cong K_q$ and $H_{vu_i,2}\cong K_r$ for any vertex $u_i$. By using the similar arguments as that of Lemma \ref{nonthing}, we may prove $|V(H_{vu_i,1})\cap V(H_{vw,2})|=0$ for $2\leqslant i\leqslant q$, and $|V(H_{vv_j,2})\cap V(H_{vw,1})|=0$ for $1\leqslant j\leqslant r$.

\begin{lemma}\label{q-1}
$|V(H_{vu_1,1})\cap V(H_{vw,1})|=q-1$.
\end{lemma}
\begin{proof}[\rm{\textbf{Proof}.}]
Since $u_1\in V(H_{vw,1})$ and $u_1\notin V(H_{vu_1,1})$, $|V(H_{vu_1,1})\cap V(H_{vw,1})|\leqslant q-1$ holds. If $|V(H_{vu_1,1})\cap V(H_{vw,1})|<q-1$, then there is an edge, say $xy$, in $G[V(H_{vu_1,1})\setminus V(H_{vw,1})]$. Furthermore, $x$, $y\notin V(H_{vw,2})$ by Lemma \ref{nonthing}. Hence, $x$, $y\notin V(H_{vw})$, and then $G[S\cup\{x,y\}]\cup H_{vw}$ is a subgraph $K_p\cup K_q\cup K_r$ in $G$. This is a contradiction.
\end{proof}

Similarly, we may prove $|V(H_{vu_i,1})\cap V(H_{vw,1})|=q-1$ for $2\leqslant i\leqslant q$, and $|V(H_{vv_j,2})\cap V(H_{vw,2})|=r-1$ for $1\leqslant j\leqslant r$. Now, we have prepared to prove Theorem \ref{pqr}.

\begin{proof}[\rm{\textbf{Proof of Theorem \ref{pqr}.}}]
Noting that $|V(H_{vu_i,1})\cap V(H_{vw,1})|=q-1$, we may let $V(H_{vu_i,1})\setminus V(H_{vw,1})=\{x_i\}$ for any $1\leqslant i\leqslant q$.

We claim that $x_1=\cdots=x_q$. Suppose to the contrary that $x_1\neq x_2$. Note that $x_2$ is adjacent to all the vertices in $V(H_{vw,1})\setminus\{u_2\}$ in $G$. Hence $u_1x_2\in E(G)$. By Lemma \ref{nonthing}, we have $x_2\notin V(H_{vw,2})$. Then the subgraph $G[S\cup\{u_1,x_2\}]\cup H_{vu_1,1}\cup H_{vw,2}$ is isomorphic to $K_p\cup K_q\cup K_r$, which is a contradiction. Hence, $x_1=\cdots=x_q:=x$, and then $G[V(H_{vw,1})\cup\{x\}]\cong K_{q+1}$.

Similarly, we may prove that there is a unique vertex $y$ such that $G[V(H_{vw,2})\cup\{y\}]\cong K_{r+1}$. Hence, we have
\begin{align*}
e(G[\overline{S}])&\geqslant e(G[V(H_{vw,1})\cup\{x\}])+e(G[V(H_{vw,2})\cup\{y\}])\\
&=\binom{q+1}{2}+\binom{r+1}{2}.
\end{align*}
By (\ref{qr}), we have
$e(G[\overline{S}])=\binom{q+1}{2}+\binom{r+1}{2}$.
Furthermore, $u_i\nsim v_j$, $x\nsim v_j$, and $y\nsim u_i$ for any $1\leqslant i\leqslant q$, $1\leqslant j\leqslant r$, and $N_G(z)=\emptyset$ for any vertex $z\notin  V(H_{vw})\cup \{x\}\cup\{y\}$.

We claim that $x\neq y$. Suppose to the contrary that $x=y$. We consider the auxiliary graph $G+u_iv_j$. The vertices with degree at least $p-1$ are in $S\cup V(H_{vw})\cup\{x\}$. While
$|S\cup V(H_{vw})\cup\{x\}|=p+q+r-1$, there is not a subgraph $K_p\cup K_q\cup K_r$ in $G+u_iv_j$.

Hence, $x\neq y$. Note that $x\nsim y$, since $e(G[\overline{S}])=\binom{q+1}{2}+\binom{r+1}{2}$. Then, by Lemma \ref{detla} (ii), we have
\begin{align*}
G&=G[S]\vee(G[V(H_{vw,1})\cup\{x\}]\cup G[V(H_{vw,2})\cup\{y\}]\cup I_{n-p-q-r})\\
&\cong K_{p-2}\vee(K_{q+1}\cup K_{r+1}\cup I_{n-p-q-r})\\
&\cong H(n;p,q,r).
\end{align*}
Then, we have
$e(G)=e(H(n;p,q,r))=(p-2)(n-p+2)+\binom{p-2}{2}+\binom{q+1}{2}+\binom{r+1}{2}.$
\end{proof}

\end{document}